\documentclass[oneside,british]{amsart}
\usepackage[T1]{fontenc}
\usepackage[utf8]{inputenc}
\usepackage{refstyle}
\usepackage{amstext}
\usepackage{amsthm}
\usepackage{amssymb}
\usepackage[authoryear]{natbib}

\makeatletter

\AtBeginDocument{\providecommand\lemref[1]{\ref{lem:#1}}}
\AtBeginDocument{\providecommand\thmref[1]{\ref{thm:#1}}}
\RS@ifundefined{subsecref}
  {\newref{subsec}{name = \RSsectxt}}
  {}
\RS@ifundefined{thmref}
  {\def\RSthmtxt{theorem~}\newref{thm}{name = \RSthmtxt}}
  {}
\RS@ifundefined{lemref}
  {\def\RSlemtxt{lemma~}\newref{lem}{name = \RSlemtxt}}
  {}

\numberwithin{equation}{section}
\numberwithin{figure}{section}
\theoremstyle{plain}
\newtheorem{thm}{\protect\theoremname}[section]
\theoremstyle{definition}
\newtheorem{defn}[thm]{\protect\definitionname}
\theoremstyle{plain}
\newtheorem{fact}[thm]{\protect\factname}
\theoremstyle{plain}
\newtheorem{lem}[thm]{\protect\lemmaname}
\theoremstyle{plain}
\newtheorem{cor}[thm]{\protect\corollaryname}
\theoremstyle{remark}
\newtheorem{rem}[thm]{\protect\remarkname}

\usepackage{braket}

\newcommand{\yes}{\mathrm{Yes}}
\newcommand{\no}{\mathrm{No}}
\newcommand{\abstain}{\mathrm{Abstain}}

\makeatother

\usepackage{babel}
\providecommand{\corollaryname}{Corollary}
\providecommand{\definitionname}{Definition}
\providecommand{\factname}{Fact}
\providecommand{\lemmaname}{Lemma}
\providecommand{\remarkname}{Remark}
\providecommand{\theoremname}{Theorem}

\begin{document}
\title{Hallucination, abstention, and computable inseparability}
\author{Takuma Imamura}
\address{Independent Researcher, AP-209, Kyoto Suzaku Studio, 2nd Floor, Kyoei
Building, 44 Sujaku Hozo-cho, Shimogyo-ku, Kyoto, 600-8846, Japan}
\email{imamura.takuma.66s@kyoto-u.jp}
\begin{abstract}
The impossibility of eliminating hallucination, understood here as
incorrect definite answers, in sufficiently expressive yes-or-no formal
domains is an immediate consequence of classical undecidability theorems.
This note does not revisit that forced-answer obstruction as its main
claim. Instead, it attempts to formally describe the corresponding
limitation for abstaining systems. Abstention can trivially avoid
hallucination if the system is allowed to abstain on every input;
the substantive question is how large the domain of guaranteed correct
non-abstaining answers can be. We formulate this question using separation
in the arithmetical hierarchy. Given disjoint sets $A$ and $B$,
any system that answers Yes on all queries indexed by $A$ and No
on all queries indexed by $B$ induces a separator of $A$ from $B$.
By combining this observation with the classical existence theorem
of $\Delta_{n}^{0}$--inseparable pairs of $\Sigma_{n}^{0}$--sets,
we yield a computability-theoretic trade-off between avoiding hallucination
by abstention and maintaining a large domain of guaranteed coverage.
\end{abstract}

\maketitle

\section{Introduction}

Large language models (LLMs) are widely used as general-purpose question-answering
systems. A common failure mode is now called \emph{hallucination}:
the system gives a fluent and seemingly authoritative answer that
is false, unsupported, or unverifiable. The term is modern, but the
phenomenon has several logically distinct components. Some are empirical
and engineering-specific, involving training data, retrieval, calibration,
reinforcement learning from human feedback, user interfaces, benchmark
design, and deployment incentives. Others are worst-case mathematical
obstructions.

This note concerns only the latter kind of obstruction. We use the
term \emph{hallucination} in a deliberately narrow formal sense: an
incorrect definite answer to a yes-or-no query. This convention is
not intended to capture every use of the term in the LLM literature.
It isolates the part of the issue that can be stated purely in computability-theoretic
terms.

The forced-answer case is already standard. If a computable system
is required to answer every sufficiently expressive yes-or-no query
by either $\yes$ or $\no$, then some answer must be wrong. This
is just the usual undecidability argument, for example the halting
problem, going back to the classical work of \citet{Chu36,Tur37}.
Recent machine-learning papers have formulated inevitability or lower-bound
results for hallucination in different settings. For example, \citet{KV24}
prove a statistical lower bound for calibrated language models, while
\citet{XJK24} formulate an inevitability claim for computable LLMs
in a computable-world framework. These results should be distinguished
from the simpler computability-theoretic point used here.

The subject of this note is not the forced-answer impossibility itself.
The subject is the trade-off that remains after allowing \emph{abstention},
a mitigation for hallucinations. An abstaining system may answer ``I
do not know,'' ``undetermined,'' ``could not solve,'' or ``abstain''
instead of giving a definite yes-or-no answer. This idea has a long
history under names such as the reject option, selective classification,
and selective prediction. It has also become central in LLM reliability
research, where abstention is studied as a way to reduce hallucination
and improve safety. See e.g. \citet{WenEtAl25}.

Abstention trivially prevents hallucination if the system abstains
on every input. Thus the substantive question is a coverage question:
on how large a class of inputs can a hallucination-free system be
guaranteed to give correct non-abstaining answers? The answer this
note provides is an immediate consequence of computable inseparability.
If the required answered domain contains both sides of a computably
inseparable pair of c.e. sets, then the system's Yes-answer set must
itself be non-computable. Such a system cannot be computable. Thus
no new theorem of computability theory is claimed here. The main aim
is to formulate the computability-theoretic trade-off between avoiding
hallucination by abstention and maintaining a large domain of guaranteed
coverage as a separation problem. This is meant to clarify the relation
between a computability-theoretic obstruction and contemporary discussions
of hallucination and abstention in artificial intelligence.

For the basic definitions and results on computability theory and
meta-mathematics, we refer to \citet{Rog87,Odi92,HP98}.

\section{Formal model of hallucination and abstention}

Fix a finite alphabet $\Sigma$ and let $\Sigma^{\ast}$ be the set
of finite strings over $\Sigma$. An AI system is idealised extensionally
as a total finite-output function. Without abstention, the output
set is
\[
\set{\yes,\no}.
\]
With abstention, the output set is
\[
\set{\yes,\no,\abstain}.
\]
The internal architecture of the system is irrelevant for the arguments
below. It may be neural, symbolic, statistical, rule-based, or hybrid;
only its input-output behaviour is used.

Let $A\subseteq\mathbb{N}$ be a definable set, and let
\[
q_{A}\colon\mathbb{N}\to\Sigma^{\ast}
\]
be a computable function that encodes the query:
\[
q_{A}\left(x\right)\colon\quad\text{Is \ensuremath{x\in A}?}
\]
For example, if $A$ is defined by an arithmetical formula $\varphi\left(x\right)$,
i.e,
\[
A=\set{x\in\mathbb{N}|\mathbb{N}\models\varphi\left(x\right)},
\]
then $q_{X}\left(x\right)$ may be the string asking whether $\varphi\left(x\right)$
holds.
\begin{defn}[Formal hallucination]
Let 
\[
F\colon\Sigma^{\ast}\to\set{\yes,\no,\abstain}
\]
be a total function (representing the extensional input--output relation
of an AI system). On the query family $q_{A}$, a \emph{hallucination}
occurs at $x\in\mathbb{N}$ if
\[
F\left(q_{A}\left(x\right)\right)=\begin{cases}
\no, & x\in A,\\
\yes, & x\notin A.
\end{cases}
\]
Note that the output $\abstain$ is not counted as a hallucination.
\end{defn}

\begin{defn}[Soundness and coverage (restricted completeness)]
The system $F$ is said to be \emph{sound} for $A$ on the query
family $q_{A}$ if it has no hallucinations on that family, that is,
\[
F\left(q_{A}\left(x\right)\right)=\yes\implies x\in A,
\]
and
\[
F\left(q_{A}\left(x\right)\right)=\no\implies x\notin A.
\]
For $D\subseteq\mathbb{N}$, the system $F$ is said to \emph{cover}
$D$ if
\[
F\left(q_{A}\left(D\right)\right)\subseteq\set{\yes,\no}.
\]
Thus, if $F$ is sound for $A$ and covers $D$, then it gives a correct
definite answer to every $q_{A}$--query indexed by $D$.
\end{defn}

The following (trivial) fact is the forced-answer case. It is noted
only as background.
\begin{fact}
Let $A\subseteq\mathbb{N}$ be a non-computable, definable set. Every
computable function
\[
F\colon\Sigma^{\ast}\to\set{\yes,\no}
\]
is unsound for $A$ on the query family $q_{A}$.
\end{fact}

The rest of the note concerns the case in which $\abstain$ is available
as an explicit output.

\section{Abstention, coverage, and inseparability}

We use the standard convention for the arithmetical hierarchy: let
$\Delta_{n}^{0}$ be the class of $0^{\left(n-1\right)}$--computable
sets, $\Sigma_{n}^{0}$ the class of $0^{\left(n-1\right)}$--computably
enumerable (c.e.) sets, $\Pi_{n}^{0}$ the class of $0^{\left(n-1\right)}$--co-c.e.
sets, where $0^{\left(n-1\right)}$ denotes the $\left(n-1\right)$--th
Turing jump.
\begin{defn}
Let $A,B\subseteq\mathbb{N}$ be disjoint sets. A set $S\subseteq\mathbb{N}$
is said to \emph{separate} $A$ from $B$ if
\[
A\subseteq S\subseteq\mathbb{N}\setminus B.
\]
For $n\geq1$, a pair $\left(A,B\right)$ of disjoint sets is called
\emph{$\Delta_{n}^{0}$--inseparable} if no $\Delta_{n}^{0}$--set
separates $A$ and $B$.
\end{defn}

\begin{fact}[cf. \citet{Odi92}]
For every $n\geq1$, there exists a pair $\left(A,B\right)$ of disjoint
$\Sigma_{n}^{0}$--sets that is not $\Delta_{n}^{0}$--separable.
(Note that the $\Pi_{n}^{0}$--set $\mathbb{N}\setminus B$ separates
$A$ and $B$.)
\end{fact}

The elementary observation behind the note is the following.
\begin{lem}[Abstaining system induces a separator]
\label{lem:induced-separator}Let $A,B\subseteq\mathbb{N}$ be disjoint
sets, $q\colon\mathbb{N}\to\Sigma^{\ast}$ a computable function,
and
\[
F\colon\Sigma^{\ast}\to\set{\yes,\no,\abstain}
\]
be a total function. Suppose that
\[
F\left(q\left(x\right)\right):=\begin{cases}
\yes, & x\in A,\\
\no,' & x\in B.
\end{cases}
\]
Then $S_{F}:=q^{-1}\left(F^{-1}\left(\yes\right)\right)$ separates
$A$ from $B$.
\end{lem}

\begin{proof}
If $x\in A$, then $F\left(q\left(x\right)\right)=\yes$, so $x\in q^{-1}\left(F^{-1}\left(\yes\right)\right)=S_{F}$.
On the other hand, if $x\in S_{F}$, then $F\left(q\left(x\right)\right)=\yes\neq\no$,
so $x\in\mathbb{N}\setminus B$. Hence
\[
A\subseteq S_{F}\subseteq\mathbb{N}\setminus B.\qedhere
\]
\end{proof}
Combining this observation with the computable inseparability theorem
gives the following coverage limitation for abstaining systems.
\begin{thm}[Coverage trade-off (total case)]
\label{thm:coverage-trade-off-total}Let $\left(A,B\right)$ be a$\Delta_{n}^{0}$--inseparable
pair $\left(A,B\right)$ of $\Sigma_{n}^{0}$--sets, $S$ a $\Pi_{n}^{0}$--separator
of $A$ from $B$, and $q_{S}$ a $0^{\left(n-1\right)}$--computable
encoding of the query ``Is $x\in S$?'', where $n\geq1$. There is
no $0^{\left(n-1\right)}$--computable function
\[
F\colon\Sigma^{\ast}\to\set{\yes,\no,\abstain}
\]
such that
\begin{description}
\item [{Soundness}] $F$ is sound for $S$ on $q_{S}$;
\item [{Coverage}] $F$ covers $A\cup B$.
\end{description}
\end{thm}

\begin{proof}
Let $F\colon\Sigma^{\ast}\to\set{\yes,\no,\abstain}$ be a function
that satisfies both the soundness and the coverage conditions. By
\lemref{induced-separator}, the set $S_{F}:=q_{S}^{-1}\left(F^{-1}\left(\yes\right)\right)$
separates $A$ from $B$. Since $A$ and $B$ are $\Delta_{n}^{0}$--inseparable,
$S_{F}$ is non-$\Delta_{n}^{0}$, and therefore $F$ is not $0^{\left(n-1\right)}$--computable.
\end{proof}
For $n=1$, this says that no (total) computable hallucination-free
abstaining system can cover both sides of a computably inseparable
pair of c.e. sets.

Abstention is treated here as an explicit observable output. We do
distinguish abstention from non-termination. If non-termination is
admitted as a separate behaviour, then one obtains a different partial-computation
problem. For instance, for disjoint c.e. sets $A$ and $B$, a computable
procedure can enumerate their elements simultaneously, output $\yes$
on elements that enter $A$, output $\no$ on elements that enter
$B$, and diverge elsewhere. However, this changes does not remove
the obstruction.
\begin{cor}[Coverage trade-off (partial case)]
Let $\left(A,B\right)$ be a$\Delta_{n}^{0}$--inseparable pair
$\left(A,B\right)$ of $\Sigma_{n}^{0}$--sets, $S$ a $\Pi_{n}^{0}$--separator
of $A$ from $B$, and $q_{S}$ a $0^{\left(n-1\right)}$--computable
encoding of the query ``Is $x\in S$?'', where $n\geq2$. There is
no $0^{\left(n-2\right)}$--computable partial function
\[
F\colon\Sigma^{\ast}\rightharpoonup\set{\yes,\no,\abstain}
\]
such that
\begin{description}
\item [{Soundness}] $F$ is sound for $S$ on $q_{S}$;
\item [{Coverage}] $F$ covers $A\cup B$.
\end{description}
\end{cor}

\begin{proof}
Let $F\colon\Sigma^{\ast}\rightharpoonup\set{\yes,\no,\abstain}$
be a partial function that satisfies both the soundness and the coverage
conditions. Suppose that $F$ is $0^{\left(n-2\right)}$--computable.
The total function defined by
\[
\tilde{F}\left(x\right):=\begin{cases}
F\left(x\right), & F\left(x\right)\downarrow,\\
\abstain, & F\left(x\right)\uparrow,
\end{cases}
\]
is $0^{\left(n-1\right)}$--computable, sound for $S$ on $q_{S}$,
and covers $A\cup B$, contradicting with \lemref{induced-separator}.
\end{proof}
For $n=2$, this implies that no partial computable hallucination-free
abstaining system can cover both sides of a $\Delta_{2}^{0}$--inseparable
pair of $\Sigma_{2}^{0}$--sets.

\section{Formal theories as examples}

Let $T$ be a consistent computably axiomatisable theory to which
the G\"odel--Rosser incompleteness applies (see e.g. \citet{Ros36,Pos44,Smu58}).
Define
\begin{align*}
A_{T} & :=\set{\left\lceil \varphi\right\rceil |T\vdash\varphi},\\
B_{T} & :=\set{\left\lceil \varphi\right\rceil |T\vdash\neg\varphi}.
\end{align*}
The set $A_{T}$ is the set of theorems of $T$, and $B_{T}$ is the
set of refutable sentences (anti-theorems) of $T$. Consistency makes
them disjoint. Computable axiomatisability makes them $\Sigma_{1}^{0}$.
The pair $\left(A_{T},B_{T}\right)$ is computably inseparable (see
\citet{Smu58}).

By \thmref{coverage-trade-off-total}, there is no pair $\left(q,F\right)$
of computable functions $q\colon\mathbb{N}\to\Sigma^{\ast}$ and $F\colon\Sigma^{\ast}\to\set{\yes,\no,\abstain}$
such that
\[
F\left(q\left(x\right)\right)=\begin{cases}
\yes, & x\in A_{T},\\
\no, & x\in B_{T}.
\end{cases}
\]

\begin{rem}
If $T$ is decidable, then both $A_{T}$ and $B_{T}$ are $\Delta_{1}^{0}$.
The inseparability obstruction above does not apply. Presburger arithmetic,
the fragment of the true arithmetic with only the addition symbol,
is a standard decidable theory (\citet{Pre29}). The first-order theory
of real closed fields is another classical decidable theory (\citet{Tar51}).
In such cases, there is, in principle, an algorithm that answers all
$T$--settled queries correctly. This theoretical possibility says
nothing about whether a deployed LLM will do so.
\end{rem}

\section{Conclusion}

In sufficiently expressive undecidable domains, any computable system
forced to give a definite answer to every query must sometimes be
wrong. This is a direct consequence of classical computability theory.
Abstention changes the obstruction into a coverage trade-off. A system
can avoid hallucination by abstaining, but computable inseparability
shows that the set of guaranteed correct non-abstaining answers may
have to be sharply restricted. If correctness is required on both
sides of a $\Delta_{n}^{0}$--inseparable pair, then a $0^{\left(n-1\right)}$--computable
finite-response system cannot provide such coverage while remaining
hallucination-free.

The preceding result is a worst-case theorem about computable input-output
functions. It does not estimate empirical hallucination rates, evaluate
calibration, design uncertainty estimators, compare abstention policies,
or analyse user incentives. Those questions are specific to machine
learning and human-computer interaction, and they have independent
scientific and practical value.

For example, abstention and selective-prediction research studies
how to trade coverage against empirical risk, how to estimate uncertainty,
and how to decide when a model should refuse to answer. In the large-language-model
setting, this includes work on truthfulness benchmarks such as TruthfulQA
\citet{LHE22}, abstention methods and evaluations surveyed by \citet{WenEtAl25},
and abstention training such as \citet{ZhoEtAl26}. Such work studies
practical mechanisms for reducing confident falsehoods in deployed
systems.

A separate practical issue is that training and evaluation procedures
may reward confident guessing rather than calibrated uncertainty.
\citet{KNVZ25} argue that common evaluation practices can incentivise
guessing over acknowledging uncertainty. Sycophancy studies also suggest
that human preference data can favour agreeable or convincingly written
answers over truthful ones in some settings \citet{SharmaEtAl24}.
Similarly, \citet{NDA24} reports experiments in which many users
prefer confident falsehoods or unmarked falsehoods to explicit admissions
of lack of knowledge. These results concern a different layer of the
problem: even where computability-theoretic obstructions are not the
binding constraint, empirical feedback loops may still amplify hallucination-like
behaviour.

These empirical studies are distinct from the computability-theoretic
impossibility proved above. Computability theory explains why perfect
coverage is impossible in worst-case formal domains. The machine-learning
literature studies how often hallucinations occur in deployed systems,
which training objectives and benchmarks amplify or mitigate them,
and how interfaces can encourage calibrated abstention rather than
confident falsehood.

\bibliographystyle{IEEEtranSN}
\bibliography{refs}

\end{document}